\documentclass{amsart}

\theoremstyle{definition}
\newtheorem{defn}{Definition}

\theoremstyle{plain}
\newtheorem{thm}[defn]{Theorem}

\newtheorem{lemma}[defn]{Lemma}
\newtheorem{cor}[defn]{Corollary}
\newtheorem{main}[defn]{Main Theorem}

\theoremstyle{remark}

\newtheorem*{acknowledgement}{Acknowledgement}

\author{Shinichiroh Matsuo}
\title{The prescribed scalar curvature problem for metrics with unit total volume}
\address{Department of Mathematics, Osaka University, Toyonaka, Osaka 560-0043, Japan}
\email{matsuo@math.sci.osaka-u.ac.jp}
\thanks{This work was supported by Grant-in-Aid for Young Scientists (B) 25800045.}

\DeclareMathOperator*{\tr}{tr}
\newcommand{\yinv}{\mu}
\newcommand{\metrics}{\mathcal{M}}

\newcommand{\dv}{d\mu}
\newcommand{\R}{\mathbb{R}}
\newcommand{\EH}{E}
\newcommand{\ric}{\mathrm{Ric}}
\renewcommand{\div}{\mathrm{div}}
\renewcommand{\tr}{\mathrm{tr}}
\renewcommand{\S}{\mathcal{S}}
\newcommand{\slice}{\Sigma}
\newcommand{\vol}{\mathrm{vol}}

\begin{document}

\begin{abstract}
	We solve the modified Kazdan-Warner problem of finding metrics with prescribed scalar curvature and unit total volume.
\end{abstract}
\maketitle

\section{Introduction}
	Kazdan and Warner~\cite{MR0375154} completely solved the prescribed scalar curvature problem on closed manifolds.
	In particular, they proved the following theorem.
	\begin{thm}[{\cite[Theorem C]{MR0375154}}] \label{thm:KW}
		Let $X$ be a closed manifold of dimension $\ge 3$.
		Every function on $X$ is the scalar curvature of some metric if and only if $X$ admits a metric of positive scalar curvature.
	\end{thm}
	
	Kobayashi~\cite{MR919505} considered the modified problem of finding metrics with prescribed scalar curvature and total volume $1$.
	Before describing his result in detail, we recall the Yamabe invariant.
	Let $X$ be a closed manifold of $\dim X = m \ge 3$ and $\metrics (X)$ the space of all smooth Riemannian metrics on $X$.
	We denote by $R_g$ or $R(g)$ the scalar curvature, by $\dv_g$ the volume form, and by $\vol(X,g)$ the volume for each metric $g \in \metrics (X)$.
	Then the normalised Einstein-Hilbert functional $\EH_X \colon \metrics (X) \to \R$ is defined as
	\[
		\EH_X \colon g \mapsto \frac{\int_X R_g \,\dv_g}{\left( \vol(X,g) \right)^{\frac{m-2}{m}}}.
	\]
	The classical Yamabe problem is to find a metric $\check{g}$ in a given conformal class $C$ such that the normalised Einstein-Hilbert functional attains its minimum on $C$: $\EH_X (\check{g}) = \inf_{g \in C} \EH_X (g) =: \yinv (X, C)$.
	This minimising metric $\check{g}$ is called a Yamabe metric, and the conformal invariant $\yinv (X, C)$ the Yamabe constant.
	Each Yamabe metric has constant scalar curvature.
	We define the topological invariant $\yinv (X)$ by the supremum of $\yinv (X, C)$ of all the conformal classes $C$ on $X$:
	\[
		\yinv (X) := \sup_{C} \yinv (X, C) = \sup_{C} \inf_{g} \frac{\int_X R_g \,\dv_g}{\left( \vol(X,g) \right)^{\frac{m-2}{m}}}
 	\]
	We call this the Yamabe invariant of $X$; it is also referred to as the $\sigma$-invariant.
	See~\cite{MR919505} and~\cite{MR994021}.
	
	A conformal class $C$ contains a metric $g \in C$ with positive scalar curvature if and only if $\yinv (X, C) > 0$, so the Yamabe invariant $\yinv (X)$ is positive if and only if $X$ admits a metric of positive scalar curvature.
	We can therefore rephrase the positive-scalar-curvature assumption in Theorem \ref{thm:KW} as the Yamabe invariant of $X$ is positive.
	
	Now we return to the modified prescribed scalar curvature problem for metrics with total volume $1$.
	In the case where the Yamabe invariant of a manifold is not positive, Kobayashi~\cite[Theorem 1]{MR919505} solved the modified problem:
	If $X$ is a closed manifold of dimension $\ge 3$ with $\yinv(X) \le 0$, then a function $f$ is the scalar curvature of some metric with total volume $1$ if and only if $f < \yinv(X)$ somewhere, unless $\yinv(X) = \yinv(X,C)$ for some conformal class $C$, in which case the constant function $f \equiv \yinv(X)$ is also attained.
	In the positive Yamabe invariant case, he obtained the following result.
	See~\cite{MR978300} for the case of even dimensional spheres.
	\begin{thm}[{\cite[Theorem 3]{MR919505}}]
		Let $X$ be a closed manifold of dimension $\ge 3$, and assume that $\yinv (X) > 0$.
		Any function that is not a constant larger than or equal to $\yinv(X)$ is the scalar curvature of some metric with total volume $1$.
	\end{thm}
	The goal of this paper is to establish the remaining cases.
	\begin{main} \label{thm:main}
		Let $X$ be a closed manifold of dimension $\ge 3$, and assume that $\yinv (X) > 0$.
		Any constant larger than or equal to $\yinv (X)$ is the scalar curvature of some metric with total volume $1$.
	\end{main}
	\begin{cor}
		Every function on a closed manifold of dimension $\ge 3$ with positive Yamabe invariant is the scalar curvature of some metric with total volume $1$.
	\end{cor}
	We have thus completely resolved the modified Kazdan-Warner problem of finding metrics with prescribed scalar curvature and total volume $1$.

\section{Proof}

\subsection{Outline}
	We first outline the proof.
	Here and subsequently, $X$ denotes a closed manifold of $\dim X = m \ge 3$ with positive Yamabe invariant.
	We emphasise that we can not directly apply the Yamabe problem because every Yamabe metric has constant scalar curvature less than or equal to that of the standard sphere when the volume is normalised~\cite{MR0431287}.
	Instead, we glue an appropriately chosen number of copies of a suitable constant positive scalar curvature metric on the sphere $S^m$ to a suitable constant positive scalar curvature metric of total volume $1$ on $X$, to obtain a metric on $X$ with constant positive scalar curvature; we then normalise the volume.
	The main difficulty in carrying out this construction is that the linearization of the scalar curvature map at a metric of constant positive scalar curvature may not be surjective.
	We make this point more precise in the following subsections.

\subsection{$V$-static metrics}
	We introduce the notion of $V$-static metrics following Corvino, Eichmair, and Miao~\cite{MR3096517}.
	Let us denote by $R \colon \metrics(X) \to C^{\infty}(X)$ the scalar curvature map and by $V \colon \metrics(X) \to (0,\infty)$ the volume map.
	The linearization $L_g$ of $R$ at a metric $g$ is
	\[
		L_g (h) = \Delta_g (\tr_g h) + \div_g \div_g h - h \cdot \ric(g),
	\]
	where $h$ is a symmetric $(0,2)$-tensor on $X$, and its formal $L^2$-adjoint is
	\[
		L_g^* (f) = (\Delta_g f) g + \nabla_g^2 f - f \ric(g).
	\]
	Our convention is that $\Delta_g f = -\tr_g (\nabla_g^2 f)$.
	The linearization of $V$ is $DV_g (h) = \frac{1}{2} \int_X \tr_g h \,\dv_g$.
	Let $\Theta(g) := (R(g), V(g))$.
	We denote by $\S_g$ its linearization $D\Theta_g(h) = (L_g(h), DV_g(h))$.
	Its formal $L^2$-adjoint is then $\S_g^* (f,a) = L_g^* f + \frac{a}{2} g$.
	We say that a smooth metric $g$ on $X$ is \emph{$V$-static} if the equation
	\[
		\S_g^* (f,a) = 0
	\]
	admits a non-trivial solution $(f,a) \in C^{\infty}(X) \times \R$.
	Note that the $V$-static condition for a metric is invariant under constant rescaling of the metric.
	
	We have the following characterisation of non-$V$-static metrics~\cite[Example 1.3]{MR3096517}.
	\begin{lemma} \label{lemma:non-v}
		Let $X$ be a closed manifold of $\dim X = m$, and $g$ be a metric with positive constant scalar curvature.
		If $g$ is not Einstein and does not admit $R_g / (m-1)$ in the positive spectrum of the Laplacian $\Delta_g$, then it is not $V$-static.
	\end{lemma}
	\begin{proof}
		Let $(f,a) \in C^{\infty}(X) \times \R$ be a solution of $S_g^* (f,a) = 0$.
		By taking trace of the equation and rearranging it, we obtain
		\[
			\Delta_g \left(f - \frac{ma}{2R_g}\right) = \frac{R_g}{m-1} \left(f - \frac{ma}{2R_g}\right).
		\]
		Recall that $R_g$ is a positive constant by assumption.
		Since $\Delta_g$ does not admit $R_g / (m-1)$ in the positive spectrum, we have $f - ma / 2R_g = 0$.
		Then, the equation $S_g^* (f,a) = 0$ again yields that $ma \ric_g = ag$.
		Since $g$ is not Einstein, it implies that $a = 0$.
		We have thus shown that $(f,a) = (0,0)$ and $g$ is not $V$-static.
	\end{proof}

\subsection{Non-$V$-static metrics with small positive scalar curvature}
	We construct a family of small constant positive scalar curvature metrics with total volume $1$, each of which is not $V$-static.
	In summary, we will show the existence of non-Ricci-flat scalar-flat metrics on a manifold with positive Yamabe invariant, and consider an appropriate deformation around them (cf.~\cite{MR0418147}).

	Koiso~\cite{MR539597} has established a local Yamabe theorem.
	We will use the following version of his theorem in the proof of Lemma~\ref{lemma:small volume} for an explicitly constructed path $\tilde{g}_t$ of smooth metrics, which turns out to be ``smooth" in the sense of ILH-manifolds and satisfies the assumption of Theorem~\ref{theorem: local Yamabe}.
	Note that, if a path of smooth metrics is smooth in the sense of ILH-manifolds, then the scalar curvature, the Ricci curvature, and the first positive eigenvalue of the Laplacian vary continously along the path.
	\begin{thm}[{\cite[Corollary 2.9]{MR539597}}] \label{theorem: local Yamabe}
		Let $X$ be a closed manifold and $\metrics(X)$ denote the ILH-manifold of all smooth Riemannian metrics on $X$. 
		Let $\tilde{g}_t$ be a smooth path in $\metrics(X)$.
		Assume that $\tilde{g}_0$ is conformal to a metric $g_0$ with zero scalar curvature.
		Then, for $|t| \ll 1$, there exists a smooth path $g_t$ in $\metrics(X)$ such that each $g_t$ is conformal to $\tilde{g}_t$, and has constant scalar curvature and total volume $1$.
	\end{thm}
	
	We now construct the desired family of small constant positive scalar curvature metrics with total volume $1$, each of which is not $V$-static.
	\begin{lemma} \label{lemma:small volume}
		Let $X$ be a closed manifold of dimension $\ge 3$ with $\yinv(X) > 0$.
		There exists a constant $\rho = \rho(X) > 0$ with the following property.
		For any $r \in (0, \rho]$, there exists a metric $g$ on $X$ with total volume $1$ that is not V-static and has the scalar curvature $R_g = r$.
	\end{lemma}
	\begin{proof}
		We will construct a path $g_t$ of non-$V$-static constant-scalar-curvature unit-total-volume metrics such that  $R(g_0) = 0$ and $R(g_t) > 0$ for $t > 0$.

		Since $X$ has positive Yamabe invariant, there exists a metric $\tilde{g}_{+1}$ with positive scalar curvature.
		On the other hand, there always exists a metric $\tilde{g}_{-1}$ with negative scalar curvature.
		Let $\tilde{g}_t = \frac{t+1}{2} \tilde{g}_{+1} + \frac{-t+1}{2} \tilde{g}_{-1}$ for $-1 \le t \le 1$.
		The path $\tilde{g}_t$ in $\metrics(X)$ is smooth in the sense of ILH-manifolds.
		Then, since the first eigenvalue of the conformal Laplacian depends continuously on $t$ and those of $\tilde{g}_{-1}$ and $\tilde{g}_{+1}$ are negative and positive respectively, we can find $\tilde{g}_{t_0}$ for some $t_0 \in (-1, 1)$ with the first eigenvalue of the conformal Laplacian equal to zero, so that $\tilde{g}_{t_0}$ is conformal to a metric with zero scalar curvature~\cite[Theorem 3.9]{MR0365409}.
		We have thus shown that the path $\tilde{g}_t$ and the metric $\tilde{g}_{t_0}$ satisfy the assumption of Theorem~\ref{theorem: local Yamabe}.
		
		We now apply Theorem~\ref{theorem: local Yamabe} to the path $\tilde{g}_t$ and the metric $\tilde{g}_{t_0}$ to obtain a new smooth path $g_t$ of constant scalar curvature smooth metrics with total volume $1$ that satisfies $R(g_0) = 0$ and $R(g_t) > 0$ for $t > 0$; furthermore, we can assume $g_t$ does not admit $R(g_t) / (m-1)$ in the positive spectrum of the Laplacian $\Delta_{g_t}$ for $t \ge 0$ because the first positive eigenvalue of the Laplacian and the scalar curvature vary continuously along the path.

		If $g_0$ is not Einstein, then $g_t$ is not Einstein for any $0 < t \ll 1$.
		Consequently, we have shown, by Lemma~\ref{lemma:non-v}, that there exists a constant $\tau > 0$ such that $g_t$ is not $V$-static for any $t \in (0,\tau]$.
		Set $\rho := R(g_{\tau})$.
		
		Assume $g_0$ is Einstein.
		The smooth diffeomorphism group of $X$ acts on the space of all smooth Riemannian metrics with constant scalar curvature and total volume $1$; its slice around $g_0$ is denoted by $\slice_{g_0}$.
		The construction of $g_0$ guarantees that the slice $\slice_{g_0}$ contains both positive and negative constant scalar curvature metrics.
		The slice $\slice_{g_0}$ itself is infinite dimensional, but Koiso~\cite[Theorem 3.1] {MR707349} has shown that there exists a finite dimensional real analytic submanifold of $\slice_{g_0}$ that contains all the Einstein metrics in $\slice_{g_0}$.
		We can find, therefore, a non-Einstein negative-scalar-curvature smooth metric $g_-$, a non-Einstein positive-scalar-curvature smooth metric $g_+$ in $\slice_{g_0}$, and a smooth path of non-Einstein smooth metrics in $\slice_{g_0}$ from $g_-$ to $g_+$
		Since the scalar curvature varies continuously along the path, there exists a non-Einstein zero-scalar-curvature smooth metric on this path.
		Now the rest of the proof runs as above.
	\end{proof}
	
	\begin{cor}
		There exists a non-Ricci-flat scalar-flat metric on any closed manifold of dimension $\ge 3$ with positive Yamabe invariant.
	\end{cor}
	
\subsection{Gluing}
	Corvino, Eichmair, and Miao~\cite{MR3096517} have established a fundamental gluing result for constant scalar curvature metrics.
	We will use the following ``non-local'' version of their theorem.
	\begin{thm}[{\cite[Theorem 1.6]{MR3096517}}] \label{thm:gluing}
		Let $(M_1, g_1)$ and $(M_2, g_2)$ be two closed $m$-dimensional Riemannian manifolds such that $R_{g_1} = R_{g_2} = m(m-1)$.
		Assume that both $g_1$ and $g_2$ are not $V$-static.
		Then, there exists a smooth metric $g$ on the connected sum $M_1 \# M_2$ such that $R_g = m(m-1)$ and $\vol (M,g) = \vol(M_1,g_1) + \vol(M_2,g_2)$.
	\end{thm}
	Rescaling metrics, we obtain the following corollary.
	\begin{cor} \label{cor:gluing}
		Let $(M_1, g_1)$ and $(M_2, g_2)$ be two closed $m$-dimensional Riemannian manifolds with constant positive scalar curvature such that $\vol(M_1,g_1) = \vol(M_2,g_2) = 1$.
		Assume that both $g_1$ and $g_2$ are not $V$-static.
		Then, there exists a constant positive scalar curvature metric $g$ on the connected sum $M_1 \# M_2$ such that $R_g^{m/2} = R_{g_1}^{m/2} + R_{g_2}^{m/2}$ and $\vol (M,g) = 1$.
	\end{cor}
	
	Now we prove our main theorem~\ref{thm:main}.
	\begin{proof}[Proof of Theorem \ref{thm:main}]
		Define $r_1 := \min \{ \rho(X), \rho(S^m) \}$, where $\rho$ is introduced in Lemma~\ref{lemma:small volume}.
		Let $r > 0$ be given.
		We express $r^{m/2} = r_0^{m/2} + (k^{2/m}r_1)^{m/2}$ in terms of $k \in \{0, 1, 2,\dots\}$ and $r_0 \in (0,r_1]$. 
		
		By Lemma~\ref{lemma:small volume}, there exist a metric $g_{r_0}$ on $X$ and a metric $g_{r_1}$ on $S^m$ such that $R(g_{r_0}) = r_0$, $R(g_{r_1}) = r_1$, $\vol(X,g_{r_0}) = \vol(S^m,g_{r_1}) = 1$, and both metrics are not $V$-static.
		In case $k > 0$, let $(M_2, g_2)$ be the $k$-disjoint union $(S^m, k^{-2/m} g_{r_1}) \sqcup \dots \sqcup (S^m, k^{-2/m} g_{r_1})$.
		Note that $R(g_2) = k^{2/m} r_1$, $\vol(M_2, g_2) = 1$, and $g_2$ is not $V$-static.
		We now apply Corollary~\ref{cor:gluing} for $(M_1, g_1) = (X, g_{r_0})$ and $(M_2, g_2)$ to obtain a metric $g$ on $X = X \# S^m \# \dots \# S^m$ such that $R_g^{m/2} = R_{g_1}^{m/2} + R_{g_2}^{m/2} = r^{m/2}$ and $\vol(X,g) = 1$.
		We have therefore constructed a metric $g$ on $X$ with $R_g = r$ and $\vol(X,g) = 1$ for any given $r > 0$.
	\end{proof}

\begin{acknowledgement}
	The author wishes to express his gratitude to O. Kobayashi for suggesting the problem and for many stimulating conversations, and to K. Akutagawa, N. Koiso, and N. Otoba for useful discussions on various aspects of this work.
	He also gratefully acknowledges the many helpful suggestions of the anonymous referee.
\end{acknowledgement}

\bibliographystyle{amsplain}
\bibliography{KW.bib}
\end{document}